\documentclass[12pt,a4paper]{amsart}
\usepackage[width=0.75\paperwidth,height=0.85\paperheight,twoside=false,includehead,includefoot]{geometry}
\allowdisplaybreaks[1]

\DeclareFontEncoding{OT2}{}{}
\DeclareFontSubstitution{OT2}{cmr}{m}{n}

\DeclareFontFamily{OT2}{cmr}{\hyphenchar\font45 }
\DeclareFontShape{OT2}{cmr}{m}{n}{%
   <5><6><7><8><9>gen*wncyr%
   <10><10.95><12><14.4><17.28><20.74><24.88>wncyr10}{}

\DeclareMathAlphabet{\mathcyr}{OT2}{cmr}{m}{n}
\DeclareMathAlphabet{\mathcyb}{OT2}{cmr}{b}{n}
\SetMathAlphabet{\mathcyr}{bold}{OT2}{cmr}{b}{n}

\usepackage{amsopn}

\DeclareMathOperator*{\Max}{Max}

\usepackage{amstext}
\usepackage{amsthm}
\usepackage{amssymb}
\usepackage{ytableau}
\usepackage{tikz}
\usepackage[all]{xy}

\newtheorem{thm}{Theorem}[section]
\newtheorem{lem}[thm]{Lemma}

\newtheorem{cor}[thm]{Corollary}

\theoremstyle{definition}
\newtheorem{defn}[thm]{Definition}
\newtheorem{ex}[thm]{Example}
\theoremstyle{remark}
\newtheorem{rem}[thm]{Remark}

\newcommand{\sh}{\mathbin{\mathcyr{sh}}}
\newcommand{\shaub}{\mathbin{\underline{\sh}}}

\begin{document}

\title{Cyclic relation for multiple zeta functions}

\author{Hideki Murahara}
\address[Hideki Murahara]{The University of Kitakyushu,  4-2-1 Kitagata, Kokuraminami-ku, Kitakyushu, Fukuoka, 802-8577, Japan}
\email{hmurahara@mathformula.page}

\author{Tomokazu Onozuka}
\address[Tomokazu Onozuka]{Institute of Mathematics for Industry, Kyushu University 744, Motooka, Nishi-ku,
Fukuoka, 819-0395, Japan}
\email{t-onozuka@imi.kyushu-u.ac.jp}


\begin{abstract}
 The cyclic relation obtained in a study by Hirose, Murakami, and the first-named author, is a wide class of relations, which includes the well-known cyclic sum formula for multiple zeta and zeta-star values, and the derivation relation for multiple zeta values. In this paper, we present its generalization to complex variables. Our proof includes a new proof of the cyclic relation. 
\end{abstract}

\maketitle

\section{Introduction}
\subsection{The Euler-Zagier multiple zeta function and functional relations}
For complex numbers $s_1,\ldots,s_r \in\mathbb{C}$, the Euler-Zagier multiple zeta function (MZF) is defined by
\begin{align*}
 \zeta(s_{1},\dots,s_{r}):=\sum_{1\le n_{1}<\cdots<n_{r}}\frac{1}{n_{1}^{s_{1}}\cdots n_{r}^{s_{r}}}.
\end{align*}
Matsumoto \cite{Mat02} proved that this series is absolutely convergent
in the domain 
\begin{align} \label{abc}
 \{(s_{1},\ldots,s_{r})\in\mathbb{C}^{r}\mid \Re(s_l+\cdots+s_r)>r-l+1\,\,\,(1\le l\le r)\}. 
\end{align}
Akiyama, Egami, and Tanigawa \cite{AET01} and Zhao \cite{Zha00} independently proved that $\zeta(s_{1},\dots,s_{r})$
can be meromorphically continued to the whole space $\mathbb{C}^{r}$. 
The special values $\zeta(k_{1},\dots,k_{r})$ with $k_1,\dots,k_{r-1}\in\mathbb{Z}_{\ge1}$ and $k_r\in\mathbb{Z}_{\ge2}$ 
of MZF are called the multiple zeta values (MZVs). 
The MZVs are real numbers and known to satisfy many kinds of algebraic relations over $\mathbb{Q}$. 

In \cite{Mat06}, Matsumoto raised the question of whether the known relations among MZVs are only valid for positive integers or not. 
In response, Ikeda and Matsuoka \cite{IM18} showed that, under certain conditions, except for the ``harmonic relation'', e.g., $\zeta(s_{1})\zeta(s_{2})=\zeta(s_{1},s_{2})+\zeta(s_{2},s_{1})+\zeta(s_{1}+s_{2})$, there are no such relations. 
Hirose and the authors obtained several complex generalizations of known MZV relations such as ``Sum formula'' and ``Ohno's relation'' (see \cite{HMO18-1} and \cite{HMO19}). 
In addition, many studies have been conducted by various mathematicians to provide functional relations (see \cite{MSV07}, \cite{Tsu07}, \cite{MT05}, \cite{MT06}, \cite{Nak06}, \cite{MNOT08}, \cite{FKMT17}, \cite{HMO18-2}, \cite{MN20}, and \cite{Kom20}, for example).

\subsection{Cyclic relation}
In this subsection, we review the cyclic analogue of MZVs (CMZVs) and introduce the cyclic relation obtained in a study by Hirose, Murakami, and the first-named author in \cite{HMM19}.
Let $d\in\mathbb{Z}_{\ge1}$, $r_{1},\dots,r_{d}\in\mathbb{Z}_{\ge1}$,
and $k_{1,1},\dots,k_{1,r_{1}},\dots,k_{d,1},\dots,k_{d,r_{d}}\in\mathbb{Z}_{\ge1}$.
A multi-index $((k_{1,1},\dots,k_{1,r_{1}}),\dots,(k_{d,1},\dots,k_{d,r_{d}}))$
is called an admissible cyclic index if 
\begin{itemize}
\item for all $1\le i\le d$, each index $(k_{i,1},\dots,k_{i,r_{i}})$
satisfies $k_{i,r_{i}}\ge2$ or equals $(1)$, 
\item there exists $1\le i\le d$ such that $(k_{i,1},\dots,k_{i,r_{i}})\neq(1)$. 
\end{itemize}
For an admissible cyclic index $\boldsymbol{k}=((k_{1,1},\dots,k_{1,r_{1}}),\dots,(k_{d,1},\dots,k_{d,r_{d}}))$,
we define the CMZV by
\begin{align*}
 \zeta^\mathrm{cyc}(\boldsymbol{k}):=\sum_{(n_{1,1},\dots,n_{d,r_{d}})\in S}\prod_{i=1}^{d}\prod_{j=1}^{r_{i}}\frac{1}{n_{i,j}^{k_{i,j}}}\label{eq:series_exp},
\end{align*}
where 
\begin{align*}
  S&
  :=\{(n_{1,1},\dots,n_{d,r_{d}})\in\mathbb{Z}_{\ge1}^{r_1+\cdots+r_d} 
   \mid 
   n_{1,1}<\cdots<n_{1,r_{1}}, \dots, n_{d,1}<\cdots<n_{d,r_{d}}, \\
  &\qquad n_{1,1} \le n_{2,r_{2}}, \dots, n_{d-1,1} \le n_{d,r_{d}}, n_{d,1} \le n_{1,r_{1}} \} \\
 &=\left\{\ (n_{1,1},\dots,n_{d,r_{d}})\in\mathbb{Z}_{\ge1}^{r_1+\cdots+r_d} \;
 \left| {\footnotesize \ytableausetup{centertableaux, boxsize=2.2em}
 \begin{ytableau}
  \none & \none & \none & \none & \none & \none & \none & \none[\;\;n_{d,r_{1}}] & \none[\quad\;\; \le n_{1,r_1}] &\none\\
  \none & \none & \none & \none & \none & \none & \none & \none[\text{\;\;\rotatebox{90}{\hspace{-15pt}$>\cdots>$}}] \\
  \none & \none & \none & \none & \none & \none & \none & \none[\le n_{d,r_{d}}]\\
  \none & \none & \none & \none & \none & \none & \none & \none \\
  \none & \none & \none & \none & \none & \none[\text{\reflectbox{$\ddots$}}] & \none  \\ 
  \none & \none & \none & \none \\
  \none & \none & \none & \none[\text{\qquad \rotatebox{90}{\hspace{-15pt}$>\cdots>$}}] \\
  \none & \none & \none[\quad n_{2,1}] & \none[\quad\;\; \le n_{3,r_3}] \\
  \none & \none & \none[\text{\quad\rotatebox{90}{\hspace{-15pt}$>\cdots>$}}] & \none \\
  \none & \none[n_{1,1}] & \none[\;\; \le n_{2,r_2}] & \none \\
  \none & \none[\text{\rotatebox{90}{\hspace{-15pt}$>\cdots>$}}] & \none & \none \\
  \none & \none[n_{1,r_1}] & \none & \none \\  
 \end{ytableau}}\
\right. \right\}.
\end{align*}
Let $\mathfrak{H}:=\mathbb{Q}\langle x,y\rangle$ and $z_{k}:=yx^{k-1}$ for $k\in\mathbb{Z}_{\geq1}$. 
We denote by $\mathfrak{H}^\mathrm{cyc}$ the subspace of $\oplus_{d=1}^{\infty}\mathfrak{H}^{\otimes d}$ 
spanned by 
\[
 \bigcup_{d=1}^{\infty}\{u_{1}\otimes\cdots\otimes u_{d}\in\mathfrak{H}^{\otimes d}\mid u_{1},\dots,u_{d}\in y\mathfrak{H}x \cup\{y\}\   
 \text{and there exists }j\ \text{such that }u_{j}\neq y\}.
\]
We define a $\mathbb{Q}$-linear map $Z^\mathrm{cyc}:\mathfrak{H}^\mathrm{cyc}\to\mathbb{R}$
by 
\[
 Z^\mathrm{cyc}(z_{k_{1,1}}\cdots z_{k_{1,r_{1}}}\otimes\cdots\otimes z_{k_{d,1}}\cdots z_{k_{d,r_{d}}})
 =\zeta^\mathrm{cyc}((k_{1,1},\dots,k_{1,r_{1}}),\dots,(k_{d,1},\dots,k_{d,r_{d}})).
\]
We also define a shuffle product $\sh:\mathfrak{H}\times\mathfrak{H}\to\mathfrak{H}$
by 
\[
1\sh w=w\sh1=w, \quad
uw\sh u'w'=u(w\sh u'w')+u'(uw\sh w'),
\]
where $u,u'\in\{x,y\}$ and $w,w'\in\mathfrak{H}$. 
Now we state the cyclic relation which is a natural extension of the well-known cyclic sum formulas and the derivation relation (see Remark \ref{remrem}). 
\begin{thm}[Cyclic relation; Hirose--Murahara--Murakami \cite{HMM19}] \label{cycrel}
 For $u_{1}\otimes\cdots\otimes u_{d}\in\mathfrak{H}^\mathrm{cyc}$, we have 
 \begin{align*}
  &\sum_{i=1}^{d}
  Z^\mathrm{cyc}(u_{1}\otimes\cdots\otimes u_{i-1}\otimes(y\shaub u_{i})\otimes u_{i+1}\otimes\cdots\otimes u_{d}) \\
  &=\sum_{i=1}^{d} 
  Z^\mathrm{cyc}(u_{1}\otimes\cdots\otimes u_{i}\otimes y\otimes u_{i+1}\otimes\cdots\otimes u_{d}),
 \end{align*} 
 where $y\shaub u_{i}=y\sh u_{i}-yu_{i}-u_{i}y$. 
\end{thm}
\begin{rem} \label{remrem}
Hoffman and Ohno \cite[Eq.(1)]{HO03} and Ohno and Wakabayashi \cite{OW06} gave clean-cut decompositions of the well-known sum formulas for MZVs and multiple zeta-star values (MZSVs), which are called the cyclic sum formulas. 
Note that these cyclic sum formulas are known to be equivalent (see \cite[Section 4]{IKOO11} and \cite[Proposition 3.3]{TW10}). 
Theorem \ref{cycrel} gives the cyclic sum formula for MZSVs if $u_i=yx^{k_i}$ for all $1\le i\le d$ (for details, see \cite[Section 5.1]{HMM19}).
Theorem \ref{cycrel} also contains the derivation relation for MZVs, which was obtained by Ihara, Kaneko, and Zagier \cite[Theorem 3]{IKZ06}, if $u_1\in y\mathfrak{H}x$ and $u_2=\cdots=u_d=y$ (for details, see \cite[Corollary 11 and Theorem 12]{HMM19}).
\end{rem}

Before concluding this subsection, we mention the number of linearly independent relations supplied 
by the cyclic sum formula \cite[Eq.(1)]{HO03}, the derivation relation \cite[Theorem 3]{IKZ06}, 
and the cyclic relation \cite[Theorem 2]{HMM19}. 
Table 1 indicates that, among MZVs, the cyclic relation provides, not all, but a large number of linearly independent relations. 
In the table, the first line means the weight of MZVs 
(we name $k:=k_1+\cdots +k_r$ the weight for $\zeta(k_1,\dots ,k_r)$). 
The second and subsequent lines of the table give the number of linearly independent relations.
The computations are performed using Mathematica. 
Note that the number of the independent relations is obtained by rewriting the original relations into MZV relations by using a well-known series representation.
\vspace{2ex}
\begin{table}[!h]
\begin{center}
\caption{Number of Independent Relations for MZVs}
\begin{tabular}{|c|r|r|r|r|r|r|r|r|r|r|r|r|r|} 
\hline
Weight & 3& 4& 5& 6& 7& 8& 9& 10 & 11 \\ 
\hline
Cyclic sum formula & 1& 2& 4& 6& 12& 18& 34& 58& 106 \\  
\hline
Derivation relation & 1& 2& 5& 10& 22& 44& 90& 181& 363 \\  
\hline
Cyclic relation & 1& 2& 5& 10& 25& 52& 110& 228& 466 \\  
\hline
All relations & 1& 3& 6& 14& 29& 60& 123& 249& 503 \\  
\hline
\end{tabular}
\end{center}
\end{table}

\subsection{Main result}
In this paper, we give a generalization of Theorem \ref{cycrel} to complex variables.  
For $i=1,\dots,d$, let $n_{i,1},\dots,n_{i,r_i}$ be positive integers and $s_{i,1},\ldots,s_{i,r_i}$ complex numbers. 
We write 
\begin{align*}
 \boldsymbol{s}_i
 &:=(s_{i,1},\dots,s_{i,r_i}), \\
 \boldsymbol{n}_i^{\boldsymbol{s}_i} 
 &:=n_{i,1}^{s_{i,1}} \cdots n_{i,r_i}^{s_{i,r_i}}, \\
 r
 &:=r_1+\cdots+r_d. 
\end{align*}
For these variables, we also write
\begin{align*}
 \boldsymbol{s}
 &:=(\boldsymbol{s}_{1},\dots,\boldsymbol{s}_{d}), \\
 \boldsymbol{n}^{\boldsymbol{s}} 
 &:=\boldsymbol{n}_{1}^{\boldsymbol{s}_{1}}\cdots\boldsymbol{n}_{d}^{\boldsymbol{s}_{d}}. 
\end{align*}
For positive integers $d$ and $r_1,\dots,r_d$, set 
\begin{align*}
 W
 &=W(r_{1},\dots,r_{d}) \\
 &:=\{ \boldsymbol{s}\in \mathbb{C}^r \mid \Re(s_{i,r_i})>1,\Re(s_{i,r_i-1}+s_{i,r_i})>2,\dots,\Re(s_{i,1}+\cdots+s_{i,r_i})>r_i, \\
 &\qquad\qquad\qquad\qquad\qquad\qquad\qquad\qquad\qquad\qquad\qquad(i=1,\dots,d \textrm{ with }r_i\ne1) \\ 
 &\qquad\qquad\qquad \Re(s_{i,r_i})\ge1\qquad\qquad\qquad\qquad\qquad\quad (i=1,\dots,d \textrm{ with }r_i=1) \}
\end{align*}
if $(r_1,\dots,r_d)\ne (1,\dots,1)$,
and set 
\begin{align*}
 W
 &=W(\underbrace{1,\dots,1}_{d}) \\
 &:=\{ \boldsymbol{s}\in \mathbb{C}^r \mid 
   \Re(s_{1,1}+\cdots+s_{d,1})>d,\\ 
   &\qquad\qquad\quad\;\;\;
   \Re(s_{l,1}+\cdots+s_{l+i,1})>i\,\,\,(1\le l\le d,\; 0\le i\le d-2)\},
\end{align*}
where we understand $s_{d+i,1}=s_{i,1}$.
Then, if $\boldsymbol{s}\in W \cap \mathbb{Z}_{\ge1}^r$, Theorem \ref{cycrel} can be written as 
\begin{align} \label{cccrel}
 \sum_{i=1}^{d} \sum_{j=1}^{r_i} \sum_{m=\delta_{j,r_i}}^{s_{i,j}-1} \sum_{S_{i,j}} 
 \frac{ n_{i,j}^{m} }{ \boldsymbol{n}^{\boldsymbol{s}} n^{m+1} }
 =\sum_{i=1}^{d} \sum_{S_i} \frac{ 1 }{ \boldsymbol{n}^{\boldsymbol{s}} n },  
\end{align} 
 where 
 \begin{align*}
  S_{i,j}&
  :=\{(n_{1,1},\dots,n_{d,r_{d}},n)\in\mathbb{Z}_{\ge1}^{r+1} 
   \mid 
   n_{1,1}<\cdots<n_{1,r_{1}}, \dots, n_{d,1}<\cdots<n_{d,r_{d}}, \\
  &\qquad 
   n_{1,1} \le n_{2,r_{2}}, \dots, n_{i-2,1} \le n_{i-1,r_{i-1}}, 
   n_{i-1,1} \le \Max\{ n_{i,r_i}, n \}, \\
  &\qquad 
   n_{i,1} \le n_{i+1,r_{i+1}}, \dots, n_{d-1,1} \le n_{d,r_{d}}, n_{d,1} \le n_{1,r_{1}}, 
   n_{i,j}<n<n_{i,j+1} \}, \\
  S_{i}&
  :=\{(n_{1,1},\dots,n_{d,r_{d}},n)\in\mathbb{Z}_{\ge1}^{r+1} 
   \mid 
   n_{1,1}<\cdots<n_{1,r_{1}}, \dots, n_{d,1}<\cdots<n_{d,r_{d}}, \\
  &\qquad n_{1,1} \le n_{2,r_{2}}, \dots, n_{d-1,1} \le n_{d,r_{d}}, n_{d,1} \le n_{1,r_{1}}, 
   n_{i,1}\le n\le n_{i+1,r_{i+1}} \}.
 \end{align*}
 Here we understand $n_{i,r_{i}+1}=\infty$ for $i=1,\dots,d$,  
 $n_{0,j}=n_{d,j}$ for $j=1,\dots,r_d$, 
 $n_{d+1,r_{d+1}}=n_{1,r_1}$, and $\delta_{j,r_i}$ denotes the Kronecker delta function (for details, see Section 3.2). 

Now, we define the functions $\widetilde\zeta_{i,j} (\boldsymbol{s})$ and $\zeta^{\mathrm{cyc}}_{i} (\boldsymbol{s})$ to give a generalization of Eq.\eqref{cccrel} to complex variables. 
Note that the functions $\widetilde\zeta_{i,j}$ and $\zeta^{\mathrm{cyc}}_{i}$ correspond to the left-hand and the right-hand sides of Eq.\eqref{cccrel} at the integer points, respectively. 
\begin{defn}
 Let $d$ and $r_1,\dots,r_d$ be positive integers. 
 For positive integers $i,j$ with $1\le i\le d$ and $1\le j\le r_i$, and $\boldsymbol{s}\in W$, we define
 \begin{align*}
  \widetilde\zeta_{i,j} (\boldsymbol{s})
  &:=\sum_{S_{i,j}} 
   \biggl(
    \frac{ n_{i,r_i}^{\delta_{j,r_i}} }{ \boldsymbol{n}^{\boldsymbol{s}} n^{\delta_{j,r_i}} (n-n_{i,j} ) }
    -\frac{ n_{i,j}^{s_{i,j}} }{ \boldsymbol{n}^{\boldsymbol{s}} n^{s_{i,j}} (n-n_{i,j} ) }
   \biggr), \\
  \zeta^{\mathrm{cyc}}_{i} (\boldsymbol{s}) 
  &:=\sum_{S_i} \frac{ 1 }{ \boldsymbol{n}^{\boldsymbol{s}} n },
 \end{align*}
 where 
 \begin{align*}
  S_{i,j}&
  :=\{(n_{1,1},\dots,n_{d,r_{d}},n)\in\mathbb{Z}_{\ge1}^{r+1} 
   \mid 
   n_{1,1}<\cdots<n_{1,r_{1}}, \dots, n_{d,1}<\cdots<n_{d,r_{d}}, \\
  &\qquad 
   n_{1,1} \le n_{2,r_{2}}, \dots, n_{i-2,1} \le n_{i-1,r_{i-1}}, 
   n_{i-1,1} \le \Max\{ n_{i,r_i}, n \}, \\
  &\qquad 
   n_{i,1} \le n_{i+1,r_{i+1}}, \dots, n_{d-1,1} \le n_{d,r_{d}}, n_{d,1} \le n_{1,r_{1}}, 
   n_{i,j}<n<n_{i,j+1} \}, \\
  S_{i}&
  :=\{(n_{1,1},\dots,n_{d,r_{d}},n)\in\mathbb{Z}_{\ge1}^{r+1} 
   \mid 
   n_{1,1}<\cdots<n_{1,r_{1}}, \dots, n_{d,1}<\cdots<n_{d,r_{d}}, \\
  &\qquad n_{1,1} \le n_{2,r_{2}}, \dots, n_{d-1,1} \le n_{d,r_{d}}, n_{d,1} \le n_{1,r_{1}}, 
   n_{i,1}\le n\le n_{i+1,r_{i+1}} \}.
 \end{align*}
 Here we understand $n_{i,r_{i}+1}=\infty$ for $i=1,\dots,d$, 
 $n_{0,j}=n_{d,j}$ for $j=1,\dots,r_d$, 
 $n_{d+1,r_{d+1}}=n_{1,r_1}$, and $\delta_{j,r_i}$ denotes the Kronecker delta function. 
\end{defn}
Then our theorem is presented as follows:
\begin{thm} \label{main}
 For $\boldsymbol{s}\in W$, we have 
 \begin{align*}
  \sum_{i=1}^{d} \sum_{j=1}^{r_i}  \widetilde\zeta_{i,j} (\boldsymbol{s})
  =\sum_{i=1}^{d} \zeta^{\mathrm{cyc}}_{i} (\boldsymbol{s}).
 \end{align*}
\end{thm}
\begin{ex}
 When $d=1$ and $r_1=r\ge2$, we have 
 \begin{align*}
  &\sum_{j=1}^{r}
  \sum_{\substack{ 1\le n_{1}<\cdots<n_{j}<n \\ n<n_{j+1}<\cdots<n_{r} } }
  \biggl(
   \frac{ n_{j}^{\delta_{j,r}} }{ n_{1}^{s_{1}} \cdots n_{r}^{s_{r}} n^{\delta_{j,r}} (n-n_{j})} 
   -\frac{ n_{j}^{s_{j}} }{ n_{1}^{s_{1}} \cdots n_{r}^{s_{r}} n^{s_{j}}(n-n_{j})}  
  \biggr) \\
  &=\sum_{\substack{ 1\le n_{1}<\cdots<n_{r} \\ n_{1}\le n \le n_{r} }}
   \frac{ 1 }{ n_{1}^{s_{1}}\cdots n_{r}^{s_{r}} n } 
 \end{align*}
 for $\Re(s_{r})>1,\Re(s_{r-1})+\Re(s_{r})>2,\dots,\Re(s_{1})+\cdots+\Re(s_{r})>r$.
\end{ex}
\begin{ex}
 When $d=1$ and $r_1=r=1$, we have 
 \begin{align*}
  \sum_{ 1\le n_{1}<n }
  \biggl(
   \frac{ n_{1} }{ n_{1}^{s} n (n-n_{1})} 
   -\frac{ 1 }{ n^{s}(n-n_{1})}  
  \biggr) 
  =\sum_{ 1\le n }
   \frac{ 1 }{ n^{s+1} } 
 \end{align*}
 for $\Re(s)>1$. 
 Note that this is equivalent to 
 \[
  \zeta_{\textrm{MT}} (s-1,1;1) -\zeta(1,s)=\zeta(s+1)
 \]
 which was obtained by Tsumura, 
 where $\zeta_{\textrm{MT}}$ is the Mordell-Tornheim double zeta function (see \cite[Proposition 2.1]{MT05}).
\end{ex}
When $r_1=\cdots=r_d=1$, we obtain the following corollary. 
\begin{cor}[Cyclic sum formula]
 For $\boldsymbol{s}\in W$, we have 
 \begin{align*}
  &\sum_{i=1}^{d} \sum_{\substack{ 1\le n_i\le\cdots\le n_d\le n_1\le \cdots \le n_{i-1} \le n \\ n\ne n_i }}
  \biggl(
   \frac{ n_{i} }{ n_1^{s_1} \cdots n_d^{s_d} n (n-n_{i} ) }
   -\frac{ n_{i}^{s_{i}} }{ n_1^{s_1}\cdots n_d^{s_d} n^{s_{i}} (n-n_{i} ) }
  \biggr) \\
  &=d\zeta(s_1+\cdots+s_d+1). 
 \end{align*}
\end{cor}
\begin{rem} \label{ooooo}
 Assume $s_1\dots,s_d\in\mathbb{Z}_{\ge1}$ in the above equality, we have the cyclic sum formula 
 for the MZSVs obtained by Ohno and Wakabayashi \cite{OW06}:
 \begin{align} \label{llmel}
  \sum_{i=1}^{d} \sum_{m=1}^{s_{i}-1}
  \zeta^{\star}(s_{i}-m,s_{i+1},\dots,s_{d},s_{1},\dots,s_{i-1},m+1)=(s_1+\cdots+s_d) \zeta(s_1+\cdots+s_d+1) 
 \end{align}
 (for details, see Section 3.2). 
\end{rem} 
\begin{rem} \label{18} 
 When $\boldsymbol{s}\in W \cap \mathbb{Z}_{\ge1}^r$, Theorem \ref{main} is equivalent to Eq.\eqref{cccrel} (see Section 3.2). 
 It is known that Eq.\eqref{cccrel} gives the derivation relation \cite[Theorem 3]{IKZ06}
 when $\boldsymbol{s}_1\in\mathbb{Z}_{\ge1}^{r_1}$ and $\boldsymbol{s}_2=(1),\dots,\boldsymbol{s}_d=(1)$ (see \cite[Section 5]{HMM19}). 
\end{rem}

\section{Proof of convergence}
\begin{lem} \label{21}
 Let $d\ge0$ and $r\ge1$. 
 For $u_1,\dots,u_d,v_1,\dots,v_r \in\mathbb{R}$, the series
 \begin{align*}
  \sum_{\substack{ 1\le m_1\le \cdots \le m_{d} \le n_r \\ 1\le n_1<\cdots<n_r }} 
  \frac{1}{ m_1^{u_1} \cdots m_{d}^{u_{d}} n_1^{v_1} \cdots n_r^{v_r} }
 \end{align*}
 is convergent when $u_1,\dots,u_d\ge1$ and $v_r>1, v_{r-1}+v_r>2, \dots, v_{1}+\cdots+v_r>r$. 
\end{lem}
\begin{proof}
 For any small $\epsilon>0$, we have
 \begin{align*}
  \sum_{1\le m_1 \le \cdots \le m_{d} \le n} 
  \frac{1}{ m_1^{u_1} \cdots m_{d}^{u_{d}} }
  \le \sum_{m_1=1}^{n} \cdots \sum_{m_d=1}^{n} \frac{1}{ m_1 \cdots m_{d} }
  \ll n^{\epsilon}. 
 \end{align*}
 Then, from Eq.\eqref{abc}, we obtain the result. 
\end{proof}

\begin{lem}
 The series 
 \begin{align*}
  \zeta^{\mathrm{cyc}} (\boldsymbol{s}) 
  &:=\sum_{S} \frac{ 1 }{ \boldsymbol{n}^{\boldsymbol{s}} }
 \end{align*}
 is absolutely convergent for $\boldsymbol{s} \in W$. 
\end{lem}
\begin{proof}
 When $r_1=\cdots=r_d=1$, we have $\zeta^{\mathrm{cyc}} (\boldsymbol{s})=\zeta(s_{1,1}+\cdots+s_{d,1})$, 
 which implies the lemma. 
 Now we consider the case $(r_1,\dots,r_d)\ne(1,\dots,1)$. 
 Note that 
 \begin{align} 
  \begin{split} \label{vvvvv}
  &\sum_{\substack{ 1\le m_1\le \cdots \le m_{d} \le n_r \\ 1\le n_1<\cdots<n_r \\
   n_1 \le m'_1\le \cdots \le m'_{d'} \le n'_{r'} \\ 1\le n'_1<\cdots<n'_{r'}
   }} 
  \frac{1}{ m_1^{u_1} \cdots m_{d}^{u_{d}} n_1^{v_1} \cdots n_r^{v_r} }
  \cdot 
  \frac{1}{ {m'}_1^{u'_1} \cdots {m'}_{d'}^{u'_{d'}} {n'}_1^{v'_1} \cdots {n'}_{r'}^{v'_{r'}} } \\
  &\le \sum_{\substack{ 1\le m_1\le \cdots \le m_{d} \le n_r \\ 1\le n_1<\cdots<n_r }} 
  \frac{1}{ m_1^{u_1} \cdots m_{d}^{u_{d}} n_1^{v_1} \cdots n_r^{v_r} }
  \sum_{\substack{ 1\le m'_1\le \cdots \le m'_{d'} \le n'_{r'} \\ 1\le n'_1<\cdots<n'_{r'} }}
  \frac{1}{ {m'}_1^{u'_1} \cdots {m'}_{d'}^{u'_{d'}} {n'}_1^{v'_1} \cdots {n'}_{r'}^{v'_{r'}} } 
  \end{split}
 \end{align}
 holds for sufficiently large $u_1,\dots,u_d,v_1,\dots,v_r,u'_1,\dots,u'_{d'},v'_1,\dots,v'_{r'}$.
 From the previous lemma and Eq.\eqref{abc}, we get the result. 
\end{proof}

\begin{lem}
 Let $d$ and $r_1,\dots,r_d$ be positive integers with $(r_1,\dots,r_d)\ne (1,\dots,1)$. 
 For positive integers $i,j$ with $1\le i\le d$ and $1\le j\le r_i$, the series 
 \begin{align*}
  \sum_{S_{i,j}} 
   \frac{ n_{i,r_i}^{\delta_{j,r_i}} }{ \boldsymbol{n}^{\boldsymbol{s}} n^{\delta_{j,r_i}} (n-n_{i,j} ) }
 \end{align*} 
 is absolutely convergent for $\boldsymbol{s}\in W$. 
\end{lem}
\begin{proof}
From Eq.\eqref{vvvvv}, it suffices to show 
 \begin{align} 
  \label{a111}
  &\sum_{\substack{ 1\le m_1\le \cdots \le m_{d} \le n_r \\ 1\le n_1<\cdots<n_i<n<n_{i+1}<\cdots<n_r }} 
  \frac{1}{ m_1^{u_1} \cdots m_{d}^{u_{d}} n_1^{v_1} \cdots n_r^{v_r} (n-n_i) }, \\
  \label{a112}
  &\sum_{\substack{ 1\le m_1\le \cdots \le m_{d} \le n \\ 1\le n_1<\cdots<n_r<n }} 
  \frac{n_r}{ m_1^{u_1} \cdots m_{d}^{u_{d}} n_1^{v_1} \cdots n_r^{v_r} n(n-n_r) }, \\
  \label{a113}
  &\sum_{\substack{ 1\le m_1\le \cdots \le m_{d} \le n \\ 1\le m_1'<n \\ 1\le m_1' \le \cdots \le m_{d'}' \le n_r \\ n_1<\cdots<n_r }} 
  \frac{ m_1' }{ m_1^{u_1} \cdots m_{d}^{u_{d}} m_1'^{u_1'} \cdots m_{d'}'^{u_{d'}'} n_1^{v_1} \cdots n_r^{v_r} n(n-m_1') }  
 \end{align}
 are convergent when $u_1,\dots,u_d, u_1',\dots,u_{d'}'\ge1$ and $v_r>1, v_{r-1}+v_r>2, \dots, v_{1}+\cdots+v_r>r$. 
 
 For Eq.\eqref{a111}, 
 since
 \begin{align*}
  \sum_{n_i<n<n_{i+1}} \frac{1}{n-n_i} 
  \ll n_{i+1}^{\epsilon} 
 \end{align*}
 for any small $\epsilon>0$, we have
 \begin{align*} 
  &\sum_{\substack{ 1\le m_1\le \cdots \le m_{d} \le n_r \\ 1\le n_1<\cdots<n_i<n<n_{i+1}<\cdots<n_r }} 
  \frac{1}{ m_1^{u_1} \cdots m_{d}^{u_{d}} n_1^{v_1} \cdots n_r^{v_r} (n-n_i) } \\
  &\ll 
  \sum_{\substack{ 1\le m_1\le \cdots \le m_{d} \le n_r \\ 1\le n_1<\cdots<n_r }} 
  \frac{ n_{i+1}^{\epsilon}  }{ m_1^{u_1} \cdots m_{d}^{u_{d}} n_1^{v_1} \cdots n_r^{v_r} }. 
 \end{align*}
 We obtain the result by using Lemma \ref{21}. 

 For Eq.\eqref{a112}, by following an argument similar to the proof in Lemma \ref{21}, we have
 \begin{align*}
  &\sum_{\substack{ 1\le m_1\le \cdots \le m_{d} \le n \\ 1\le n_1<\cdots<n_r<n }} 
  \frac{n_r}{ m_1^{u_1} \cdots m_{d}^{u_{d}} n_1^{v_1} \cdots n_r^{v_r} n(n-n_r) } \\
  &\ll \sum_{ 1\le n_1<\cdots<n_r<n } 
  \frac{n_r n^{\epsilon} }{ n_1^{v_1} \cdots n_r^{v_r} n(n-n_r) }. 
 \end{align*}  
 Since 
 \begin{align*}
  \sum_{n_r<n} \frac{1}{n^{1-\epsilon}(n-n_r)} 
  &=\sum_{n=1}^{\infty} \frac{1}{(n+n_r)^{1-\epsilon}n} 
  \ll \frac{1}{n_r^{1-2\epsilon}}, 
 \end{align*}
 we have
 \begin{align*}
  \sum_{ 1\le n_1<\cdots<n_r<n } 
  \frac{n_r n^{\epsilon} }{ n_1^{v_1} \cdots n_r^{v_r} n(n-n_r) } 
  \ll \sum_{ 1\le n_1<\cdots<n_r } 
  \frac{ n_r^{2\epsilon} }{ n_1^{v_1} \cdots n_r^{v_r} }.
 \end{align*}
 By Eq.\eqref{abc}, we have the result. 

 For Eq.\eqref{a113}, when $u_1'>1$, it is easy to check the convergence 
 by using Lemma \ref{21}, Eq.\eqref{vvvvv}, and the convergence of Eq.\eqref{a112}. 
 When $u_1'=1$, by following an argument similar to the proof in Lemma \ref{21}, 
 we have
 \begin{align*}
  &\sum_{\substack{ 1\le m_1\le \cdots \le m_{d} \le n \\ 1\le m_1'<n \\ 1\le m_1' \le \cdots \le m_{d'}' \le n_r \\ n_1<\cdots<n_r }} 
  \frac{ m_1' }{ m_1^{u_1} \cdots m_{d}^{u_{d}} m_1'^{u_1'} \cdots m_{d'}'^{u_{d'}'} n_1^{v_1} \cdots n_r^{v_r} n(n-m_1') } \\
  &\ll \sum_{\substack{ 1\le m_1'<n \\ 1\le m_1' \le \cdots \le m_{d'}' \le n_r \\ n_1<\cdots<n_r }} 
  \frac{ m_1' n^{\epsilon} }{ m_1'^{u_1'} \cdots m_{d'}'^{u_{d'}'} n_1^{v_1} \cdots n_r^{v_r} n(n-m_1') }.
 \end{align*}
 Since     
 \begin{align*}
  \sum_{1 \le m_1' \le n_r}\sum_{ m_1'<n } \frac{ m_1' }{ m_1' n^{1-\epsilon} (n-m_1') } 
  \ll \sum_{1 \le m_1' \le n_r} \frac{1}{m_1'^{1-2\epsilon} } 
  \ll n_r^{2\epsilon}, 
 \end{align*}
 we have
 \begin{align*}
  &\sum_{\substack{ 1\le m_1'<n \\ 1\le m_1' \le \cdots \le m_{d'}' \le n_r \\ n_1<\cdots<n_r }} 
  \frac{ m_1' n^{\epsilon} }{ m_1'^{u_1'} \cdots m_{d'}'^{u_{d'}'} n_1^{v_1} \cdots n_r^{v_r} n(n-m_1') } \\
  &\ll \sum_{\substack{ 1\le m_2' \le \cdots \le m_{d'}' \le n_r \\ n_1<\cdots<n_r }} 
  \frac{ n_r^{2\epsilon} }{ m_2'^{u_2'} \cdots m_{d'}'^{u_{d'}'} n_1^{v_1} \cdots n_r^{v_r} }. 
 \end{align*}
Following Lemma \ref{21}, this completes the proof. 
\end{proof}

\begin{lem}
 Let $d$ and $r_1,\dots,r_d$ be positive integers with $(r_1,\dots,r_d)\ne (1,\dots,1)$. 
 For positive integers $i,j$ with $1\le i\le d$ and $1\le j\le r_i$, the series 
 \begin{align*}
  \sum_{S_{i,j}} 
   \frac{ n_{i,j}^{s_{i,j}} }{ \boldsymbol{n}^{\boldsymbol{s}} n^{s_{i,j}} (n-n_{i,j} ) }
 \end{align*} 
 is absolutely convergent for $\boldsymbol{s}\in W$. 
\end{lem}
\begin{proof}
 Following Eq.\eqref{vvvvv}, it suffices to show that
 \begin{align} 
  \label{b111}
  &\sum_{\substack{ 1\le m_1\le \cdots \le m_{d} \le n_r \\ 1\le n_1<\cdots<n_i<n<n_{i+1}<\cdots<n_r }} 
  \frac{ n_i^{v_i} }{ m_1^{u_1} \cdots m_{d}^{u_{d}} n_1^{v_1} \cdots n_r^{v_r} n^{v_i} (n-n_i) }, \\
  \label{b112}
  &\sum_{\substack{ 1\le m_1\le \cdots \le m_{d} \le n \\ 1\le n_1<\cdots<n_r<n }} 
  \frac{n_r^{v_r}}{ m_1^{u_1} \cdots m_{d}^{u_{d}} n_1^{v_1} \cdots n_r^{v_r} n^{v_r} (n-n_r) }, \\
  \label{b113}
  &\sum_{\substack{ 1\le m_1\le \cdots \le m_{d} \le n \\ 1\le m_1'<n \\ 1\le m_1' \le \cdots \le m_{d'}' \le n_r \\ n_1<\cdots<n_r }} 
  \frac{ m_1'^{u_1'} }{ m_1^{u_1} \cdots m_{d}^{u_{d}} m_1'^{u_1'} \cdots m_{d'}'^{u_{d'}'} n_1^{v_1} \cdots n_r^{v_r} n^{u_1'} (n-m_1') }  
 \end{align}
 are convergent when $u_1,\dots,u_d, u_1',\dots,u_{d'}'\ge1$ and $v_r>1, v_{r-1}+v_r>2, \dots, v_{1}+\cdots+v_r>r$. 
 
 For Eq.\eqref{b111}, 
 since
 \begin{align*}
  \sum_{n_i<n} \frac{1}{n-n_i} 
  \ll n^{\epsilon} 
 \end{align*}
 for any small $\epsilon>0$, we have
 \begin{align*} 
  &\sum_{\substack{ 1\le m_1\le \cdots \le m_{d} \le n_r \\ 1\le n_1<\cdots<n_i<n<n_{i+1}<\cdots<n_r }} 
  \frac{ n_i^{v_i} }{ m_1^{u_1} \cdots m_{d}^{u_{d}} n_1^{v_1} \cdots n_r^{v_r} n^{v_i} (n-n_i) } \\
  &\ll 
  \sum_{\substack{ 1\le m_1\le \cdots \le m_{d} \le n_r \\ 1\le n_1<\cdots<n_{i-1}<n<n_{i+1}<\cdots<n_r }} 
  \frac{ n_{i}^{v_i} n^{\epsilon}  }{ m_1^{u_1} \cdots m_{d}^{u_{d}} n_1^{v_1} \cdots n_r^{v_r} n^{v_i} }. 
 \end{align*}
 By following Lemma \ref{21}, we find the result. 

 The series Eq.\eqref{b112} and Eq.\eqref{b113} are bounded by Eq.\eqref{a112} and Eq.\eqref{a113}, respectively. 
\end{proof}

\begin{lem}
 Let $d$ be a positive integer and $r_1=\dots=r_d=1$. 
 For a positive integer $i$ with $1\le i\le d$, the series 
 \begin{align*}
  \sum_{S_{i,1}} \frac{ n_{i,1} }{ \boldsymbol{n}^{\boldsymbol{s}} n (n-n_{i,1} ) }
  &=\sum_{\substack{ n_{i,1} \le\cdots\le n_{d,1} \le n_{1,1} \le\cdots\le n_{i-1,1}\le n \\ n\ne n_{i,1} }} 
   \frac{ n_{i,1} }{ \boldsymbol{n}^{\boldsymbol{s}} n (n-n_{i,1} ) }, \\
  \sum_{S_{i,1}} \frac{ n_{i,1}^{s_{i,1}} }{ \boldsymbol{n}^{\boldsymbol{s}} n^{s_{i,1}} (n-n_{i,1} ) } 
  &=\sum_{\substack{ n_{i,1} \le\cdots\le n_{d,1} \le n_{1,1} \le\cdots\le n_{i-1,1}\le n \\ n\ne n_{i,1} }} 
   \frac{ n_{i,1}^{s_{i,1}} }{ \boldsymbol{n}^{\boldsymbol{s}} n^{s_{i,1}} (n-n_{i,1} ) }
 \end{align*}
 are absolutely convergent for $\boldsymbol{s}\in W$.
\end{lem}
\begin{proof}
 We prove that the first series converges absolutely by induction on $d$. 
 When $d=1$, we easily see the convergence (see \cite[Theorem 2.2]{OO15}, for example).
 Assume that the first series converges absolutely when $d-1$. 
 By the induction hypothesis, we need only to consider the series
 \[
  \sum_{ n_{i,1} <\cdots< n_{d,1} < n_{1,1} <\cdots< n_{i-1,1}\le n } 
   \frac{ n_{i,1} }{ \boldsymbol{n}^{\boldsymbol{s}} n (n-n_{i,1} ) }. 
 \]
 Since
 \[
  \sum_{ n_{i-1,1}\le n } \frac{ 1 }{ n(n-n_{i,1}) }
  \ll \frac{ 1 }{ n_{i-1,1}^{1-\epsilon} } 
 \]
 for any $\epsilon>0$, we have
 \begin{align*}
  &\sum_{ n_{i,1} <\cdots< n_{d,1} < n_{1,1} <\cdots< n_{i-1,1}\le n } 
   \frac{ n_{i,1} }{ \boldsymbol{n}^{\boldsymbol{s}} n (n-n_{i,1} ) } \\
  & \ll \sum_{ n_{i,1} <\cdots< n_{d,1} < n_{1,1} <\cdots< n_{i-1,1} } 
   \frac{ 1 }{ n_{i,1}^{\Re(s_{i,1})-1} n_{i+1,1}^{\Re (s_{i+1,1})} \cdots n_{i-2,1}^{\Re (s_{i-2,1})}  n_{i-1,1}^{\Re (s_{i-1,1})+1-\epsilon} }. 
 \end{align*}
 Following Eq.\eqref{abc}, we have the result. 
 Similarly, we can check that the second series converges absolutely. 
\end{proof}

\section{Proof and applications of Theorem \ref{main}}
\subsection{Proof of Theorem \ref{main}}
We define
\begin{align*}
 \widetilde\zeta_{i,j}^{(1)} (\boldsymbol{s})
 &:=\sum_{S_{i,j}}
  \frac{ n_{i,r_i}^{\delta_{j,r_i}} }{ \boldsymbol{n}^{\boldsymbol{s}} n^{\delta_{j,r_i}} (n-n_{i,j} ) }, \\
 \widetilde\zeta_{i,j}^{(2)} (\boldsymbol{s})
 &:=\sum_{S_{i,j}}
  \frac{ n_{i,j}^{s_{i,j}} }{ \boldsymbol{n}^{\boldsymbol{s}} n^{s_{i,j}} (n-n_{i,j} ) }. 
\end{align*}

\begin{lem} \label{31}
 For $\boldsymbol{s}\in W$, we have 
 \begin{align*}
  \widetilde\zeta_{i,j}^{(1)} (\boldsymbol{s})
  =
  \begin{cases}
   \displaystyle 
   \sum_{S} \frac{1}{\boldsymbol{n}^{\boldsymbol{s}}} \left( 1+\frac{1}{2}+\cdots+\frac{1}{n_{i,j+1}-n_{i,j}-1} \right) 
   &(j\ne r_i), \\
   \displaystyle 
   \sum_{T_i} \frac{1}{\boldsymbol{n}^{\boldsymbol{s}}} \left( \frac{1}{\max\{ 1, n_{i-1,1}-n_{i,r_i} \}}+\cdots
    +\frac{1}{\max\{ n_{i,r_i}, n_{i-1,1}-1 \}} \right) 
    &(j= r_i),
  \end{cases}
 \end{align*} 
 where 
 \begin{align*}
  T_{1}&:=\{(n_{1,1},\dots,n_{d,r_{d}})\in\mathbb{Z}_{\ge1}^{r} 
   \mid 
   n_{1,1}<\cdots<n_{1,r_{1}}, \dots, n_{d,1}<\cdots<n_{d,r_{d}}, \\
  &\qquad\qquad\qquad\qquad\qquad\qquad\qquad\qquad\qquad\qquad\qquad\,\,\, n_{1,1} \le n_{2,r_{2}}, \dots, n_{d-1,1} \le n_{d,r_{d}} \}, \\
  T_{i}&:=\{(n_{1,1},\dots,n_{d,r_{d}})\in\mathbb{Z}_{\ge1}^{r} 
   \mid 
   n_{1,1}<\cdots<n_{1,r_{1}}, \dots, n_{d,1}<\cdots<n_{d,r_{d}}, \\
  &\qquad n_{1,1} \le n_{2,r_{2}}, \dots, n_{i-2,1} \le n_{i-1,r_{i-1}}, 
   n_{i,1} \le n_{i+1,r_{i+1}}, \dots, n_{d-1,1} \le n_{d,r_{d}}, n_{d,1} \le n_{1,r_{1}} \} \\
   &\qquad\qquad\qquad\qquad\qquad\qquad\qquad\qquad\qquad\qquad\qquad\qquad\qquad\qquad\qquad\qquad\quad (i\ne1).
 \end{align*}
\end{lem}
\begin{proof}
 When $j\ne r_i$, we have
 \begin{align*}
  \widetilde\zeta_{i,j}^{(1)} (\boldsymbol{s})
  &=\sum_{\substack{ S \\ n_{i,j}<n<n_{i,j+1} }} 
   \frac{ 1 }{ \boldsymbol{n}^{\boldsymbol{s}}  (n-n_{i,j} ) } \\
  &=\sum_{S} \frac{1}{\boldsymbol{n}^{\boldsymbol{s}}} \left( 1+\frac{1}{2}+\cdots+\frac{1}{n_{i,j+1}-n_{i,j}-1} \right). 
 \end{align*}
Conversely, when $j=r_i$, we have
 \begin{align*}
  \widetilde\zeta_{i,j}^{(1)} (\boldsymbol{s})
  &=\sum_{\substack{ T_i \\ n_{i,r_i}<n \\ n_{i-1,1} \le n }} 
   \frac{ n_{i,r_i} }{ \boldsymbol{n}^{\boldsymbol{s}} n(n-n_{i,j} ) } \\
  &=\sum_{\substack{ T_i \\ n_{i,r_i}<n \\ n_{i-1,1} \le n }} 
   \frac{ 1 }{ \boldsymbol{n}^{\boldsymbol{s}} } \left( \frac{1}{n-n_{i,j}} -\frac{1}{n} \right) \\
  &=\sum_{T_i} \frac{1}{\boldsymbol{n}^{\boldsymbol{s}}} \left( \frac{1}{\max\{ 1, n_{i-1,1}-n_{i,r_i} \}}+\cdots
    +\frac{1}{\max\{ n_{i,r_i}, n_{i-1,1}-1 \}} \right). \qedhere
 \end{align*}
\end{proof}

\begin{lem} \label{32}
 For $\boldsymbol{s}\in W$, we have 
 \begin{align*}
  \widetilde\zeta_{i,j}^{(2)} (\boldsymbol{s})
  =
  \begin{cases}
   \displaystyle 
   \sum_{T_{i+1}} \frac{1}{\boldsymbol{n}^{\boldsymbol{s}}} \left( \frac{1}{\max\{ 1, n_{i,1}-n_{i+1,r_{i+1}} \}}+\cdots
    +\frac{1}{ n_{i,1}-1 } \right) \quad (j=1), \\
   \displaystyle 
   \sum_{S} \frac{1}{\boldsymbol{n}^{\boldsymbol{s}}} \left( 1+\frac{1}{2}+\cdots+\frac{1}{n_{i,j}-n_{i,j-1}-1} \right) \quad (j\ne 1),
  \end{cases}
 \end{align*} 
\end{lem}
\begin{proof}
 When $j=1$, we have
 \begin{align*}
  \widetilde\zeta_{i,j}^{(2)} (\boldsymbol{s})
  &=\sum_{\substack{ T_{i+1} \\ n<n_{i,1} \\ n \le n_{i+1,r_{i+1}} }} 
   \frac{ 1 }{ \boldsymbol{n}^{\boldsymbol{s}} (n_{i,1}-n) } \\
  &=\sum_{T_{i+1}} \frac{1}{\boldsymbol{n}^{\boldsymbol{s}}} \left( \frac{1}{\max\{ 1, n_{i,1}-n_{i+1,r_{i+1}} \}}+\cdots
    +\frac{1}{ n_{i,1}-1 } \right). 
 \end{align*}
  When $j\ne1$, we have
 \begin{align*}
  \widetilde\zeta_{i,j}^{(2)} (\boldsymbol{s})
  &=\sum_{\substack{ S \\ n_{i,j-1}<n<n_{i,j} }} 
   \frac{ 1 }{ \boldsymbol{n}^{\boldsymbol{s}}  (n_{i,j}-n ) } \\
  &=\sum_{S} \frac{1}{\boldsymbol{n}^{\boldsymbol{s}}} \left( 1+\frac{1}{2}+\cdots+\frac{1}{n_{i,j}-n_{i,j-1}-1} \right). \qedhere
 \end{align*}
\end{proof}

\begin{proof}[Proof of Theorem \ref{main}]
 Note that
 \[
  \widetilde\zeta_{i,j}^{(1)} (\boldsymbol{s}) -\widetilde\zeta_{i,j+1}^{(2)} (\boldsymbol{s})=0
 \]
 holds for $j=1,\dots,r_i-1$.
 Then we have
 \begin{align*}
  \sum_{i=1}^{d} \sum_{j=1}^{r_i} \widetilde\zeta_{i,j} (\boldsymbol{s})
  =\sum_{i=1}^{d} ( \widetilde\zeta_{i,r_i}^{(1)} (\boldsymbol{s}) -\widetilde\zeta_{i,1}^{(2)} (\boldsymbol{s}) ). 
 \end{align*}
 By Lemmas \ref{31} and \ref{32}, we have 
 \begin{align*}
  \sum_{i=1}^{d} \sum_{j=1}^{r_i}  \widetilde\zeta_{i,j} (\boldsymbol{s})
  &=\sum_{i=1}^{d} \sum_{T_{i+1}} \frac{1}{\boldsymbol{n}^{\boldsymbol{s}}} \left( \frac{1}{n_{i,1}}+\cdots
   +\frac{1}{ \max\{ n_{i+1,r_{i+1}},n_{i,1}-1 \} } \right) \\  
  &=\sum_{i=1}^{d} \sum_{\substack{ T_{i+1} \\ n_{i,1} \le n \le n_{i+1,r_{i+1}} }} 
   \frac{1}{ \boldsymbol{n}^{\boldsymbol{s}} n}. 
 \end{align*}
 This completes the proof. 
\end{proof}

\subsection{Remarks on Theorems \ref{cycrel} and \ref{main}, Eq.\eqref{cccrel}, and the cyclic sum formula }
First, we show the equivalence of Theorem \ref{cycrel} and Eq.\eqref{cccrel}. 
Hereinafter, we write $\boldsymbol{s}=\boldsymbol{k}$ if $\boldsymbol{s}\in W \cap \mathbb{Z}_{\ge1}^r$. 
Using $u_i:=z_{k_{i,1}}\cdots z_{k_{i,r_i}}$ for $i=1,\dots,d$, we have
\[
 y\shaub u_{i}
 =\sum_{j=1}^{r_i} \sum_{m=\delta_{j,r_i}}^{k_{i,j}-1} 
 z_{k_{i,1}}\cdots z_{k_{i,j-1}} 
 z_{k_{i,j}-m} z_{m+1}
 z_{k_{i,j+1}}\cdots z_{k_{i,r_i}}. 
\]
Then we have
\begin{align*}
 &\sum_{i=1}^{d}
  Z^\mathrm{cyc}(u_{1}\otimes\cdots\otimes u_{i-1}\otimes(y\shaub u_{i})\otimes u_{i+1}\otimes\cdots\otimes u_{d}) \\
 &=\sum_{i=1}^{d} \sum_{j=1}^{r_i} \sum_{m=\delta_{j,r_i}}^{k_{i,j}-1}
  Z^\mathrm{cyc}(u_{1}\otimes\cdots\otimes u_{i-1}\otimes
 z_{k_{i,1}}\cdots z_{k_{i,j-1}} 
 z_{k_{i,j}-m} z_{m+1}
 z_{k_{i,j+1}}\cdots z_{k_{i,r_i}} \\ 
  &\qquad\qquad\qquad\qquad\qquad
 \otimes u_{i+1}\otimes\cdots\otimes u_{d}) \\
 &=\sum_{i=1}^{d} \sum_{j=1}^{r_i} \sum_{m=\delta_{j,r_i}}^{k_{i,j}-1} \sum_{S_{i,j}} 
  \frac{ n_{i,j}^{m} }{ \boldsymbol{n}^{\boldsymbol{k}} n^{m+1} }. 
\end{align*} 
Since
\begin{align*}
 &\sum_{i=1}^{d} 
  Z^\mathrm{cyc}(u_{1}\otimes\cdots\otimes u_{i}\otimes y\otimes u_{i+1}\otimes\cdots\otimes u_{d})
 =\sum_{i=1}^{d} \sum_{S_i} \frac{ 1 }{ \boldsymbol{n}^{\boldsymbol{k}} n }, 
\end{align*}
Theorem \ref{cycrel} and Eq.\eqref{cccrel} are found to be equivalent. 
 
The correspondence between the function $\widetilde\zeta_{i,j}$ and the left-hand side of Eq.\eqref{cccrel} is shown as follows: 
\begin{align*}
 \sum_{i=1}^{d} \sum_{j=1}^{r_i} \widetilde\zeta_{i,j} (\boldsymbol{k}) 
 &=\sum_{i=1}^{d} \sum_{j=1}^{r_i} \sum_{S_{i,j}} 
   \biggl(
    \frac{ n_{i,r_i}^{\delta_{j,r_i}} }{ \boldsymbol{n}^{\boldsymbol{k}} n^{\delta_{j,r_i}} (n-n_{i,j} ) }
    -\frac{ n_{i,j}^{k_{i,j}} }{ \boldsymbol{n}^{\boldsymbol{k}} n^{k_{i,j}} (n-n_{i,j} ) }
   \biggr) \\
 &=\sum_{i=1}^{d} \sum_{j=1}^{r_i} \sum_{S_{i,j}} 
   \biggl(
    \sum_{m\ge0} \frac{ n_{i,j}^{m+\delta_{j,r_i}} }{ \boldsymbol{n}^{\boldsymbol{k}} n^{m+1+\delta_{j,r_i}} }
    -\sum_{m\ge0} \frac{ n_{i,j}^{s_{i,j}+m} }{ \boldsymbol{n}^{\boldsymbol{k}} n^{k_{i,j}+m+1} }
   \biggr) \\
 &=\sum_{i=1}^{d} \sum_{j=1}^{r_i} \sum_{m=\delta_{j,r_i}}^{k_{i,j}-1} \sum_{S_{i,j}} 
 \frac{ n_{i,j}^{m} }{ \boldsymbol{n}^{\boldsymbol{k}} n^{m+1} }.
\end{align*}
Hence Theorems \ref{cycrel} and \ref{main}, and Eq.\eqref{cccrel} are equivalent at the integer points. 

Last, we explain that the cyclic relation implies the cyclic sum formula for MZSVs mentioned in Remark \ref{ooooo}. 
Note that the following explanation is similar to \cite[Section 5.1]{HMM19}.
Assume $r_1=\cdots=r_d=1$. 
From Eq.\eqref{cccrel}, for $\boldsymbol{k}\in W \cap \mathbb{Z}_{\ge1}^r$, we have
\begin{align*} 
 \sum_{i=1}^{d} \sum_{m=1}^{k_{i}-1} 
  \sum_{\substack{ n_{i} \le\cdots\le n_{d} \le n_{1} \le\cdots\le n_{i-1}\le n \\ n\ne n_{i} }} 
  \frac{ n_{i}^{m} }{ \boldsymbol{n}^{\boldsymbol{k}} n^{m+1} } 
 =\sum_{i=1}^{d} 
  \sum_{ n_1\le \cdots \le n_i\le n\le n_{i+1} \le \cdots\le n_d\le n_1 } 
  \frac{ 1 }{ \boldsymbol{n}^{\boldsymbol{k}} n }.  
\end{align*}
The left-hand side of the above equality equals
\begin{align*}
 \sum_{i=1}^{d} \sum_{m=1}^{k_{i}-1} 
 \left( 
  \zeta^{\star}(k_{i}-m,k_{i+1},\dots,k_{d},k_{1},\dots,k_{i-1},m+1)
  -\zeta (k_1+\cdots+k_d+1) 
 \right). 
\end{align*}
The right-hand side of the above equality equals
\begin{align*}
 d \zeta (k_1+\cdots+k_d+1). 
\end{align*}
Then we have Eq.\eqref{llmel}.

\section*{Acknowledgment}
This work was supported by JSPS KAKENHI Grant Number JP19K14511.


\end{document}